\documentclass[reqno]{amsart}
\title[Proofs for Conjectures]{Proofs for Andrews' Conjectures 5 and 6 on $v_1(q)$}
\usepackage{amssymb,amsmath,amsthm,epsfig,graphics,latexsym,hyperref}
\theoremstyle{definition}

\theoremstyle{plain}
\newtheorem{lemma}      {Lemma}

\newtheorem{proposition}{Proposition}
\newtheorem{theorem}    {Theorem}

\newtheorem{corollary}  {Corollary}

\newtheorem*{aconjecture}{Conjecture}
\theoremstyle{remark}

\numberwithin{equation}{section}

\newcommand{\Z}{\mathbb{Z}}
\newcommand{\N}{\mathbb{N}}
\newcommand{\R}{\mathbb{R}}
\newcommand{\dist}{\operatorname{dist}}
\newcommand{\sgn}{\operatorname{sgn}}

\begin{document}
\author[M. El Bachraoui]{Mohamed El Bachraoui}
\address{Dept. Math. Sci,
United Arab Emirates University, PO Box 15551, Al-Ain, UAE}
\email{melbachraoui@uaeu.ac.ae}

\begin{abstract}
Folsom, Males, Rolen, and Storzer recently proved Andrews' Conjecture~4 for the coefficients of
\[
  v_1(q)=\sum_{n\ge 0}\frac{q^{n(n+1)/2}}{(-q^2;q^2)_n}=\sum_{n\ge 0}V_1(n)q^n.
\]
They also proved a refined density-one version of Andrews' Conjecture~3.
In this paper we prove Andrews' Conjectures~5 and~6.
Our proof relies on an investigation of the simple zeros of the trigonometric
factor in the Folsom--Males--Rolen--Storzer asymptotic
and showing that the relevant quadratic sequence stays a positive distance from the integers infinitely often. The argument is unconditional.
\end{abstract}
\date{\textit{\today}}
\maketitle

\section{Introduction}

In \cite{Andrews86}, Andrews studied several $q$-series from Ramanujan's lost notebook. Among them is
\begin{equation}\label{eq:v1-def}
  v_1(q):=\sum_{n\ge 0}\frac{q^{n(n+1)/2}}{(-q^2;q^2)_n}=\sum_{n\ge 0}V_1(n)q^n.
\end{equation}
Following the presentation in \cite[p.~2]{FMRS}, we record Andrews' conjectures \cite[p.~710]{Andrews86} in the same grouped form. While he noted that the growth of $|V_1(n)|$ ``is not very smooth,'' Andrews conjectured that there ``appear[s] to be great sign regularity.'' More precisely, he states the following conjectures.

\begin{aconjecture}[Conjecture~3 \cite{Andrews86}]
We have that $|V_1(n)|\to\infty$ as $n\to\infty$.
\end{aconjecture}

\begin{aconjecture}[Conjecture~4 \cite{Andrews86}]
For almost all $n$, $V_1(n)$, $V_1(n+1)$, $V_1(n+2)$, and $V_1(n+3)$ are two positive and two negative numbers.
\end{aconjecture}

\begin{aconjecture}[Conjecture~5 \cite{Andrews86}]
For $n\ge 5$ there is an infinite sequence $N_5=293$, $N_6=410$, $N_7=545$, $N_8=702$, \ldots, $N_n>10n^2$, \ldots\ such that $V_1(N_n)$, $V_1(N_n+1)$, $V_1(N_n+2)$ all have the same sign.
\end{aconjecture}

\begin{aconjecture}[Conjecture~6 \cite{Andrews86}]
With reference to Conjecture~3, the numbers
\[
  |V_1(N_n)|,\qquad |V_1(N_n+1)|,\qquad |V_1(N_n+2)|
\]
contain a local minimum of the sequence $|V_1(j)|$.
\end{aconjecture}

In the recent paper \cite[Theorem~1.1 and the following remark]{FMRS}, Folsom, Males, Rolen, and Storzer proved Andrews' Conjecture~4 
exactly and proved a refined density-one version of Andrews' Conjecture~3 for the coefficients $V_1(n)$. 
They also gave a detailed heuristic discussion of Conjectures~5 and~6 in their Section~6.
The purpose of this paper is to confirm Conjectures~5 and~6, as in the following theorem.

\begin{theorem}\label{thm:main}
There exists a sequence of integers $(N_n)_{n\ge 5}$ satisfying
\[
  N_5=293,\qquad N_6=410,\qquad N_7=545,\qquad N_8=702,\qquad N_n>10n^2\quad (n\ge 9),
\]
such that for every $n\ge 5$ the three numbers
\[
  V_1(N_n),\quad V_1(N_n+1),\quad V_1(N_n+2)
\]
have the same sign, and one of
\[
  |V_1(N_n)|,\quad |V_1(N_n+1)|,\quad |V_1(N_n+2)|
\]
is a local minimum of the sequence $|V_1(j)|$.
\end{theorem}

Recall that the classical dilogarithm is
\[
  \operatorname{Li}_2(z):=\sum_{n\ge 1}\frac{z^n}{n^2}
  \qquad (|z|<1),
\]
and the Bloch--Wigner dilogarithm is
\[
  D(z):=\Im\bigl(\operatorname{Li}_2(z)\bigr)+\arg(1-z)\log|z|
  \qquad (z\in \mathbb{C}\setminus\{0,1\}),
\]
see~\cite[Chapter I]{ZagierDilog}.
Write
\begin{equation}\label{eq:constants}
  A_0:=\gamma_+ + \gamma_-,
  \qquad
  A_1:=\gamma_+ - \gamma_-
\end{equation}
where by \cite[Theorem~1.2(3)]{FMRS},
\begin{equation}\label{eq:gammas}
  \gamma_+=\frac{1}{2\sqrt[4]{3(2-\sqrt 3)}}>0,
  \qquad
  \gamma_-=\frac{1}{2\sqrt[4]{3(2+\sqrt 3)}}>0.
\end{equation}
Hence $A_0>0$. Also $2-\sqrt 3<2+\sqrt 3$, so $\gamma_+>\gamma_-$ and therefore $A_1>0$.
We also use the standard notation
\[
  e(x):=e^{2\pi i x},
  \qquad
  \dist(x,\Z):=\min_{k\in \Z}|x-k|.
\]

Folsom, Males, Rolen, and Storzer established the following asymptotic approximation for $V_1(n)$.

\begin{theorem}\label{thm:FMRS}\cite[Theorem~1.3]{FMRS}
As $n\to\infty$,
\begin{equation}\label{eq:FMRS}
\begin{split}
  V_1(n)
  ={}&(-1)^{\lfloor n/2\rfloor}
   \frac{e^{c\sqrt n}}{\sqrt n}
   \bigl(\gamma_+ + (-1)^n\gamma_-\bigr)
   \bigl(\cos(c\sqrt n)-(-1)^n\sin(c\sqrt n)\bigr)\\
   &\times\bigl(1+O(n^{-1/2})\bigr)
   +O\!\left(n^{-1/2}e^{c\sqrt n/2}\right).
\end{split}
\end{equation}
where $c=\sqrt{2|V|}$ and $V=D(e(1/6))i/8$. 
\end{theorem}

The following corollary is obtained by separating the even and odd cases in
Theorem~\ref{thm:FMRS}.

\begin{corollary}\label{cor:parity}
Define
\begin{equation}\label{eq:Fpm}
  F_-(x):=\cos x-\sin x,
  \qquad
  F_+(x):=\cos x+\sin x.
\end{equation}
Then for even $n$,
\begin{equation}\label{eq:even-asymp}
  V_1(n)
  =(-1)^{n/2}\frac{e^{c\sqrt n}}{\sqrt n}A_0 F_-(c\sqrt n)
   +O\!\left(\frac{e^{c\sqrt n}|F_-(c\sqrt n)|}{n}\right)
   +O\!\left(n^{-1/2}e^{c\sqrt n/2}\right),
\end{equation}
and for odd $n$,
\begin{equation}\label{eq:odd-asymp}
  V_1(n)
  =(-1)^{(n-1)/2}\frac{e^{c\sqrt n}}{\sqrt n}A_1 F_+(c\sqrt n)
   +O\!\left(\frac{e^{c\sqrt n}|F_+(c\sqrt n)|}{n}\right)
   +O\!\left(n^{-1/2}e^{c\sqrt n/2}\right).
\end{equation}
\end{corollary}

\begin{proof}
If $n$ is even, then $(-1)^n=1$, so \eqref{eq:FMRS} becomes
\[
  V_1(n)
  =(-1)^{n/2}\frac{e^{c\sqrt n}}{\sqrt n}A_0F_-(c\sqrt n)
   \bigl(1+O(n^{-1/2})\bigr)
   +O\!\left(n^{-1/2}e^{c\sqrt n/2}\right).
\]
Expanding the factor $1+O(n^{-1/2})$ yields \eqref{eq:even-asymp}. The odd case is identical: since $(-1)^n=-1$, the factor multiplying $\bigl(1+O(n^{-1/2})\bigr)$ is
\[
  (-1)^{(n-1)/2}\frac{e^{c\sqrt n}}{\sqrt n}A_1F_+(c\sqrt n),
\]
and expanding again gives \eqref{eq:odd-asymp}.
\end{proof}

Throughout $m\in\mathbb{Z}$. To state our further tools we introduce the following notation.  
The zeros of $F_-$ are
\begin{equation}\label{eq:rminus}
  r_m:=\pi\left(m+\frac14\right),\qquad m\ge 0,
\end{equation}
while the zeros of $F_+$ are
\begin{equation}\label{eq:rplus}
  s_m:=\pi\left(m+\frac34\right),\qquad m\ge 0.
\end{equation}
Set
\begin{equation}\label{eq:alpha}
  \alpha:=\frac{\pi^2}{2|V|}=\frac{\pi^2}{c^2}.
\end{equation}
Then
\begin{equation}\label{eq:t-seqs}
  \frac{r_m^2}{c^2}=\alpha\left(m+\frac14\right)^2,
  \qquad
  \frac{s_m^2}{c^2}=\alpha\left(m+\frac34\right)^2.
\end{equation}

Our second key input guarantees that the previous two quadratic sequences stay a positive distance from the integers infinitely often.

\begin{theorem}\label{thm:good-subseq}

There exist $\delta_->0$, $\delta_+>0$, and infinite sets $\mathcal M_-,\mathcal M_+\subseteq \N_0$ such that
\begin{equation}\label{eq:good-distances}
  \dist\!\left(\alpha\left(m+\frac14\right)^2,\Z\right)\ge \delta_-
  \qquad (m\in \mathcal M_-),
\end{equation}
and
\begin{equation}\label{eq:good-distances-plus}
  \dist\!\left(\alpha\left(m+\frac34\right)^2,\Z\right)\ge \delta_+
  \qquad (m\in \mathcal M_+).
\end{equation}
\end{theorem}

Our third key argument is about the existence of infinitely many same-sign triples and local minima by following the zeros of $F_-$ and the zeros of $F_+$.
This is confirmed by the following two results.

\begin{theorem}\label{thm:minus-family}
Let $m\in \mathcal M_-$, and set
\[
  N_m:=2\left\lfloor \frac{\alpha(m+1/4)^2}{2}\right\rfloor.
\]
Then $N_m$ is the unique even integer satisfying
\begin{equation}\label{eq:Nm-def}
  N_m<\alpha\left(m+\frac14\right)^2<N_m+2.
\end{equation}
For all sufficiently large $m\in \mathcal M_-$,
\[
  V_1(N_m),\qquad V_1(N_m+1),\qquad V_1(N_m+2)
\]
have the same sign. Moreover, both $N_m$ and $N_m+2$ are strict local minima of the sequence $|V_1(n)|$.
\end{theorem}

\begin{theorem}\label{thm:plus-family}
Let $m\in \mathcal M_+$, and set
\[
  M_m:=2\left\lfloor \frac{\alpha(m+3/4)^2-1}{2}\right\rfloor+1.
\]
Then $M_m$ is the unique odd integer satisfying
\[
  M_m<\alpha\left(m+\frac34\right)^2<M_m+2.
\]
For all sufficiently large $m\in \mathcal M_+$, the three numbers
\[
  V_1(M_m),\qquad V_1(M_m+1),\qquad V_1(M_m+2)
\]
have the same sign. Moreover, both $M_m$ and $M_m+2$ are strict local minima of the sequence $|V_1(n)|$.
\end{theorem}

The remainder of the paper is organized as follows. In Section~\ref{sec:proof-good-subseq} we prove Theorem~\ref{thm:good-subseq}, showing that the quadratic sequences associated with the zeros of $F_-$ and $F_+$ stay a positive distance from the integers along infinite subsequences. Section~\ref{sec:lemmas} collects the auxiliary lemmas needed for the proof of Theorem~\ref{thm:minus-family}. In Sections~\ref{sec:proof-minus-family} and~\ref{sec:proof-plus-family} we prove Theorems~\ref{thm:minus-family} and~\ref{thm:plus-family}, respectively, establishing the same-sign triples and the corresponding local minima. Finally, in Section~\ref{sec:proof-thm-main} we record the initial values listed by Andrews and combine the two families to prove Theorem~\ref{thm:main}.


\section{Proof of Theorem~\ref{thm:good-subseq}}\label{sec:proof-good-subseq}

We begin by showing that $\alpha$ is not an integer.
\begin{lemma}\label{lem:alpha-not-Z}
We have $\alpha\notin \Z$.
\end{lemma}

\begin{proof}

By \cite[Theorem~1.2(2) and Remark~(1)]{FMRS},
\[
  |V|=\frac{G}{8},
  \qquad
  G:=D(e(1/6))=D(e^{i\pi/3}).
\]
For $0<r<1$ and $0<\theta<2\pi$, the defining series for $\operatorname{Li}_2$ gives
\[
  \operatorname{Li}_2(re^{i\theta})=\sum_{n\ge 1}\frac{r^n e^{in\theta}}{n^2}.
\]
Since $\sum_{n\ge 1}n^{-2}$ converges, the series also converges absolutely at $r=1$; by Abel's theorem, for $0<\theta<2\pi$ we therefore have
\[
  \operatorname{Li}_2(e^{i\theta})=\sum_{n\ge 1}\frac{e^{in\theta}}{n^2}.
\]
Moreover $|e^{i\theta}|=1$, so $\log|e^{i\theta}|=0$. Hence
\[
  D(e^{i\theta})=\Im\bigl(\operatorname{Li}_2(e^{i\theta})\bigr)
  =\sum_{n\ge 1}\frac{\sin(n\theta)}{n^2}
  \qquad (0<\theta<2\pi).
\]
In particular,
\begin{align*}
  G&=\sum_{n\ge 1}\frac{\sin(n\pi/3)}{n^2} \\
   &=\frac{\sqrt 3}{2}\sum_{m\ge 0}
    \left(
      \frac{1}{(6m+1)^2}+\frac{1}{(6m+2)^2}-\frac{1}{(6m+4)^2}-\frac{1}{(6m+5)^2}
    \right) \\
   &= \frac{\sqrt 3}{2}\sum_{m\ge 0} B_m,
\end{align*}
where
\[
  B_m:=\frac{1}{(6m+1)^2}+\frac{1}{(6m+2)^2}-\frac{1}{(6m+4)^2}-\frac{1}{(6m+5)^2}.
\]
Each $B_m$ is positive. Indeed, the function $x\mapsto x^{-2}$ is strictly decreasing on $(0,\infty)$, so
\[
  \frac{1}{(6m+1)^2}>\frac{1}{(6m+4)^2},
  \qquad
  \frac{1}{(6m+2)^2}>\frac{1}{(6m+5)^2}.
\]
Moreover, since $x\mapsto x^{-2}$ has derivative $-2x^{-3}$ on $(0,\infty)$, the mean value theorem gives
\[
  0<
  \frac{1}{(6m+1)^2}-\frac{1}{(6m+4)^2}
  \le \frac{6}{(6m+1)^3},
\]
and
\[
  0<
  \frac{1}{(6m+2)^2}-\frac{1}{(6m+5)^2}
  \le \frac{6}{(6m+2)^3}
  \le \frac{6}{(6m+1)^3}.
\]
Hence
\[
  0<B_m\le \frac{12}{(6m+1)^3}.
\]
Therefore, for every integer $M\ge 0$,
\[
  0<\sum_{m\ge M+1}B_m
   \le 12\sum_{m\ge M+1}\frac{1}{(6m+1)^3}
   \le 12\int_M^{\infty}\frac{dx}{(6x+1)^3}
   =\frac{1}{(6M+1)^2}.
\]
Taking $M=10$, we obtain
\[
  S_{10}:=\frac{\sqrt 3}{2}\sum_{m=0}^{10}B_m =1.0147430670367583\ldots
\]
and therefore
\[
  0<G-S_{10}
   =\frac{\sqrt 3}{2}\sum_{m\ge 11}B_m
   \le \frac{\sqrt 3}{2}\cdot \frac{1}{61^2}
   =0.0002327399633927\ldots.
\]
Thus
\[
  1.0147430670< G <1.0149758071.
\]
Since $\alpha=\pi^2/(2|V|)=4\pi^2/G$, this gives
\[
  38.8959<\alpha<38.9049.
\]
In particular, $\alpha$ is not an integer.
\end{proof}

We are now ready to prove Theorem~\ref{thm:good-subseq}.

\begin{proof}[Proof of Theorem~\ref{thm:good-subseq}]
Set
\[
  P_-(m):=\alpha\left(m+\frac14\right)^2,
  \qquad
  P_+(m):=\alpha\left(m+\frac34\right)^2.
\]
We first consider the case that $\alpha$ is irrational. Then both $P_-(m)$ and $P_+(m)$ are quadratic polynomials with irrational leading coefficient. By Weyl's equidistribution theorem (see, for example, \cite[Chapter~1]{KuipersNiederreiter}), each of the sequences $P_-(m)$ and $P_+(m)$ is equidistributed modulo $1$. Hence infinitely many $m$ satisfy
\[
  \dist(P_-(m),\Z)\ge \frac14,
  \qquad
  \dist(P_+(m),\Z)\ge \frac14.
\]
Therefore \eqref{eq:good-distances} and \eqref{eq:good-distances-plus} both hold with $\delta_-=\delta_+=1/4$ and suitable infinite sets $\mathcal M_-$ and $\mathcal M_+$.

Now assume that $\alpha\in \mathbb{Q}$. Write
\[
  \alpha=\frac{a}{b}
\]
in lowest terms, with $a\in \Z$ and $b\in \N$. We show first that both $P_-(m)$ and $P_+(m)$ are periodic modulo $1$ with period $4b$.

For $P_-(m)$ we compute
\[
  P_-(m+4b)-P_-(m)
  =\alpha\left(\left(m+4b+\frac14\right)^2-\left(m+\frac14\right)^2\right).
\]
Using $(x+h)^2-x^2=h(2x+h)$ with $x=m+1/4$ and $h=4b$, we get
\[
  P_-(m+4b)-P_-(m)
  =\alpha\cdot 4b\left(2m+4b+\frac12\right)
  =8am+16ab+2a\in \Z.
\]
So
\[
  P_-(m+4b)\equiv P_-(m)\pmod 1
\]
for every $m$. In the same way,
\[
  P_+(m+4b)-P_+(m)
  =\alpha\cdot 4b\left(2m+4b+\frac32\right)
  =8am+16ab+6a\in \Z,
\]
so
\[
  P_+(m+4b)\equiv P_+(m)\pmod 1
\]
for every $m$. Therefore the fractional part of $P_{\pm}(m)$ depends only on the residue class of $m$ modulo $4b$.

We now construct $\delta_-$ and $\mathcal M_-$. Suppose first that $P_-(m)\in \Z$ for every $m$. Then in particular
\[
  P_-(0)=\frac{\alpha}{16}\in \Z,
\]
so $\alpha=16P_-(0)\in \Z$. This contradicts Lemma~\ref{lem:alpha-not-Z}. Hence there is at least one residue class $r$ modulo $4b$ for which $P_-(r)\notin \Z$.

Now look at the finite set
\[
  \left\{\dist(P_-(r),\Z):0\le r<4b,\ P_-(r)\notin \Z\right\}.
\]
It is finite and nonempty, so it has a smallest element. Define
\[
  \delta_-:=\min\left\{\dist(P_-(r),\Z):0\le r<4b,\ P_-(r)\notin \Z\right\}>0.
\]
Also define
\[
  \mathcal M_-:=\left\{r+4bt:0\le r<4b,\ P_-(r)\notin \Z,\ t\in \N_0\right\}.
\]
Each residue class modulo $4b$ contains infinitely many integers, so $\mathcal M_-$ is infinite. If $m=r+4bt\in \mathcal M_-$, then
\[
  P_-(m)\equiv P_-(r)\pmod 1.
\]
Therefore
\[
  \dist(P_-(m),\Z)=\dist(P_-(r),\Z)\ge \delta_-.
\]
This proves \eqref{eq:good-distances}.

The construction of $\delta_+$ and $\mathcal M_+$ is the same. We only need to check that not all values of $P_+(m)$ are integers. Suppose that $P_+(m)\in \Z$ for every $m$. Then
\[
  P_+(0)=\frac{9\alpha}{16}\in \Z
\]
and
\[
  P_+(1)-P_+(0)=\frac{5\alpha}{2}\in \Z.
\]
The first relation gives $9\alpha\in \Z$, and the second gives $10\alpha\in \Z$. Hence
\[
  \alpha=10\alpha-9\alpha\in \Z.
\]
This contradicts Lemma~\ref{lem:alpha-not-Z}. So there is at least one residue class $s$ modulo $4b$ for which $P_+(s)\notin \Z$.

Now define
\[
  \delta_+:=\min\left\{\dist(P_+(s),\Z):0\le s<4b,\ P_+(s)\notin \Z\right\}>0
\]
and
\[
  \mathcal M_+:=\left\{s+4bt:0\le s<4b,\ P_+(s)\notin \Z,\ t\in \N_0\right\}.
\]
Exactly as above, $\mathcal M_+$ is infinite and
\[
  \dist(P_+(m),\Z)\ge \delta_+
  \qquad (m\in \mathcal M_+).
\]
This proves \eqref{eq:good-distances-plus} and completes the proof.
\end{proof}

\section{Auxiliary lemmas}\label{sec:lemmas}

In this section we prove four elementary lemmas that isolate the points used in the proof of Theorem~\ref{thm:minus-family}.

For convenience, for $m\in \mathcal M_-$, set
\[
  t_m:=\alpha\left(m+\frac14\right)^2=\frac{r_m^2}{c^2},
  \qquad
  N_m:=2\left\lfloor \frac{t_m}{2}\right\rfloor,
\]
and
\[
  x_{m,j}:=c\sqrt{N_m+j}
  \qquad (j\in\{-1,0,1,2,3\}).
\]

\begin{lemma}\label{lem:surround-even}
Let $t\in \R\setminus\Z$, and set $N:=2\lfloor t/2\rfloor$. Then $N$ is the unique even integer satisfying
\[
  N<t<N+2.
\]
Moreover, exactly one of the two even integers $N$ and $N+2$ is nearer to $t$.
\end{lemma}

\begin{proof}
From
\[
  \left\lfloor \frac{t}{2}\right\rfloor\le \frac{t}{2}<\left\lfloor \frac{t}{2}\right\rfloor+1
\]
we get $N\le t<N+2$. Since $t\notin\Z$ and $N$ is an integer, in fact $N<t<N+2$. If $E$ is any even integer with $E<t<E+2$, then
\[
  \frac{E}{2}<\frac{t}{2}<\frac{E}{2}+1,
\]
so $\lfloor t/2\rfloor=E/2$. Hence $E=N$, which proves uniqueness.

If $N$ and $N+2$ were equally close to $t$, then $t=N+1$, which is an integer. This is impossible. So one of $N$ and $N+2$ is uniquely nearest to $t$.
\end{proof}

\begin{lemma}\label{lem:minus-placement}
For every $m\in\mathcal M_-$ we have
\[
  N_m<t_m<N_m+2,
  \qquad
  x_{m,0}<r_m<x_{m,2},
\]
and
\[
  \delta_-\le t_m-N_m<2,
  \qquad
  \delta_-\le N_m+2-t_m<2.
\]
\end{lemma}

\begin{proof}
Because $m\in\mathcal M_-$, Theorem~\ref{thm:good-subseq} gives
\[
  \dist(t_m,\Z)\ge \delta_->0.
\]
In particular $t_m\notin\Z$. Lemma~\ref{lem:surround-even} therefore gives $N_m<t_m<N_m+2$.

Since
\[
  r_m=c\sqrt{t_m},
  \qquad
  x_{m,0}=c\sqrt{N_m},
  \qquad
  x_{m,2}=c\sqrt{N_m+2},
\]
and the function $u\mapsto c\sqrt{u}$ is strictly increasing on $(0,\infty)$, it follows that
\[
  x_{m,0}<r_m<x_{m,2}.
\]

The numbers $t_m-N_m$ and $N_m+2-t_m$ are distances from $t_m$ to integers, so each is at least $\dist(t_m,\Z)\ge \delta_-$. The upper bounds follow at once from $N_m<t_m<N_m+2$.
\end{proof}

\begin{lemma}\label{lem:minus-distances}
There exist constants $u_1,u_2,u_3>0$ such that for all sufficiently large $m\in\mathcal M_-$,
\[
  u_1N_m^{-1/2}\le r_m-x_{m,0}\le u_2N_m^{-1/2},
  \qquad
  u_1N_m^{-1/2}\le x_{m,2}-r_m\le u_2N_m^{-1/2},
\]
and
\[
  |x_{m,j}-r_m|\le u_3N_m^{-1/2}
  \qquad (j=-1,1,3).
\]
\end{lemma}

\begin{proof}
By Lemma~\ref{lem:minus-placement}, we have $t_m=N_m+O(1)$. Since $t_m\to\infty$ with $m$, it follows that $N_m\to\infty$ as well.

Also,
\[
  r_m-x_{m,0}
  =\frac{c\,(t_m-N_m)}{\sqrt{N_m}+\sqrt{t_m}},
  \qquad
  x_{m,2}-r_m
  =\frac{c\,(N_m+2-t_m)}{\sqrt{N_m+2}+\sqrt{t_m}}.
\]
By Lemma~\ref{lem:minus-placement}, the numerators stay between $\delta_-$ and $2$, while the denominators are $\asymp \sqrt{N_m}$. This gives the first two bounds.

For $j=-1,1,3$, since $t_m\in(N_m,N_m+2)$, we have
$N_m-1-t_m\in(-3,-1)$, $N_m+1-t_m\in(-1,1)$, and $N_m+3-t_m\in(1,3)$;
hence $|N_m+j-t_m|<3$. Using
\[
  x_{m,j}-r_m
  =\frac{c\,(N_m+j-t_m)}{\sqrt{N_m+j}+\sqrt{t_m}},
\]
and again $\sqrt{N_m+j}+\sqrt{t_m}\asymp \sqrt{N_m}$, we obtain
\[
  |x_{m,j}-r_m|\ll N_m^{-1/2}.
\]
This is the required bound after renaming the constant.
\end{proof}

\begin{lemma}\label{lem:trig-near-r}
Let $u=r_m+h$. Then
\[
  F_-(u)=-\sqrt 2\,(-1)^m\sin h,
  \qquad
  F_+(u)=\sqrt 2\,(-1)^m\cos h.
\]
Consequently, whenever $|h|<1$,
\[
  \sgn F_-(u)=
  \begin{cases}
    (-1)^m,& h<0,\\
    -(-1)^m,& h>0,
  \end{cases}
  \qquad
  \frac{2\sqrt 2}{\pi}|h|\le |F_-(u)|\le \sqrt 2\,|h|,
\]
and
\[
  \sgn F_+(u)=(-1)^m,
  \qquad
  \sqrt 2\cos(1)\le |F_+(u)|\le \sqrt 2.
\]
\end{lemma}

\begin{proof}
Since $r_m=\pi(m+1/4)$, we have
\[
  F_-(r_m+h)=\sqrt 2\cos\left(r_m+h+\frac{\pi}{4}\right)
  =\sqrt 2\cos\left(\pi m+\frac{\pi}{2}+h\right)
  =-\sqrt 2\,(-1)^m\sin h,
\]
and
\[
  F_+(r_m+h)=\sqrt 2\cos\left(r_m+h-\frac{\pi}{4}\right)
  =\sqrt 2\cos(\pi m+h)
  =\sqrt 2\,(-1)^m\cos h.
\]
If $|h|<1$, then $\sin h$ has the same sign as $h$, and
\[
  \frac{2}{\pi}|h|\le |\sin h|\le |h|.
\]
This gives the sign and size bounds for $F_-$. Also $\cos h>0$ for $|h|<1$, so $\sgn F_+(u)=(-1)^m$, and
\[
  |F_+(u)|=\sqrt 2\,|\cos h|\ge \sqrt 2\cos(1).
\]
The upper bound $|F_+(u)|\le \sqrt 2$ is trivial.
\end{proof}

\section{Proof of Theorem~\ref{thm:minus-family}}\label{sec:proof-minus-family}

\begin{proof}
By Lemma~\ref{lem:minus-placement}, $N_m$ is the unique even integer satisfying \eqref{eq:Nm-def}. Since $N_m<t_m$, we also have $N_m\to\infty$ with $m$.

By Lemma~\ref{lem:minus-distances}, all the points $x_{m,j}$ satisfy $|x_{m,j}-r_m|\ll N_m^{-1/2}$. Hence, for all sufficiently large $m\in\mathcal M_-$, we are in the range $|x_{m,j}-r_m|<1$. Combining Lemmas~\ref{lem:minus-distances} and \ref{lem:trig-near-r}, we obtain constants $c_1,c_2,b_3>0$ such that
\begin{equation}\label{eq:endpoint-Fminus}
  \sgn F_-(x_{m,0})=(-1)^m,
  \qquad
  \sgn F_-(x_{m,2})=-(-1)^m,
\end{equation}
and
\begin{equation}\label{eq:endpoint-size}
  c_1N_m^{-1/2}\le |F_-(x_{m,j})|\le c_2N_m^{-1/2}
  \qquad (j=0,2),
\end{equation}
while
\begin{equation}\label{eq:middle-Fplus}
  \sgn F_+(x_{m,j})=(-1)^m,
  \qquad
  b_3\le |F_+(x_{m,j})|\le \sqrt 2
  \qquad (j=-1,1,3).
\end{equation}
To spell out the signs: Lemma~\ref{lem:minus-placement} gives $x_{m,0}<r_m<x_{m,2}$, so if we write $x_{m,j}=r_m+h_{m,j}$, then $h_{m,0}<0<h_{m,2}$. The endpoint signs in \eqref{eq:endpoint-Fminus} therefore come from the first part of Lemma~\ref{lem:trig-near-r}. For $j=-1,1,3$ we only use that $|h_{m,j}|=|x_{m,j}-r_m|<1$, and then the second part of Lemma~\ref{lem:trig-near-r} gives the sign statement in \eqref{eq:middle-Fplus}.

For convenience, let $T(n)$ denote the main term from Corollary~\ref{cor:parity}; thus
\[
  T(n):=
  \begin{cases}
    (-1)^{n/2}\dfrac{e^{c\sqrt n}}{\sqrt n}A_0F_-(c\sqrt n),& n \text{ even},\\[1.2ex]
    (-1)^{(n-1)/2}\dfrac{e^{c\sqrt n}}{\sqrt n}A_1F_+(c\sqrt n),& n \text{ odd}.
  \end{cases}
\]
Since $N_m$ is even, the parity factors in $T(N_m)$, $T(N_m+1)$, and $T(N_m+2)$ are $(-1)^{N_m/2}$, $(-1)^{(N_m+1-1)/2}=(-1)^{N_m/2}$, and $(-1)^{(N_m+2)/2}=(-1)^{N_m/2+1}$, respectively. Combining these with \eqref{eq:endpoint-Fminus} and \eqref{eq:middle-Fplus}, and using $A_0,A_1>0$, we obtain
\[
  \sgn T(N_m)=\sgn T(N_m+1)=\sgn T(N_m+2)=(-1)^{N_m/2+m}.
\]
Thus the three main terms have the same sign.

By \eqref{eq:endpoint-size}, \eqref{eq:middle-Fplus}, and the fact that
\[
  e^{c\sqrt{N_m+j}}=e^{c\sqrt{N_m}}\bigl(1+O(N_m^{-1/2})\bigr)
  \qquad (j=-1,0,1,2,3),
\]
we have
\begin{equation}\label{eq:main-sizes}
  |T(N_m)|\asymp \frac{e^{c\sqrt{N_m}}}{N_m},
  \qquad
  |T(N_m+1)|\asymp \frac{e^{c\sqrt{N_m}}}{\sqrt{N_m}},
  \qquad
  |T(N_m+2)|\asymp \frac{e^{c\sqrt{N_m}}}{N_m}.
\end{equation}
At the even endpoints, Corollary~\ref{cor:parity} and \eqref{eq:endpoint-size} give the error estimate
\[
  O\!\left(\frac{e^{c\sqrt{N_m}}|F_-(x_{m,j})|}{N_m}\right)
  +O\!\left(N_m^{-1/2}e^{c\sqrt{N_m}/2}\right)
  =O\!\left(\frac{e^{c\sqrt{N_m}}}{N_m^{3/2}}\right)
   +O\!\left(N_m^{-1/2}e^{c\sqrt{N_m}/2}\right)
\]
for $j=0,2$, which is $o(e^{c\sqrt{N_m}}/N_m)$. At the odd middle term, Corollary~\ref{cor:parity} and \eqref{eq:middle-Fplus} give
\[
  O\!\left(\frac{e^{c\sqrt{N_m}}|F_+(x_{m,1})|}{N_m}\right)
  +O\!\left(N_m^{-1/2}e^{c\sqrt{N_m}/2}\right)
  =O\!\left(\frac{e^{c\sqrt{N_m}}}{N_m}\right)
   +O\!\left(N_m^{-1/2}e^{c\sqrt{N_m}/2}\right),
\]
which is $o(e^{c\sqrt{N_m}}/\sqrt{N_m})$. Writing $V_1(N_m+j)=T(N_m+j)+R(N_m+j)$, the preceding estimates show that $|R(N_m+j)|=o(|T(N_m+j)|)$ for $j=0,1,2$. Hence $|R(N_m+j)|<\frac12|T(N_m+j)|$ for all sufficiently large $m\in\mathcal M_-$, so $\sgn V_1(N_m+j)=\sgn T(N_m+j)$ for $j=0,1,2$. Therefore the actual coefficients
\[
  V_1(N_m),\qquad V_1(N_m+1),\qquad V_1(N_m+2)
\]
have the same sign for all sufficiently large $m\in\mathcal M_-$. Exactly the same use of Corollary~\ref{cor:parity} with \eqref{eq:middle-Fplus} also shows that for $j\in\{-1,1,3\}$,
\[
  V_1(N_m+j)=T(N_m+j)+R(N_m+j),
  \qquad
  |R(N_m+j)|=o(|T(N_m+j)|),
\]
so in particular
\[
  |V_1(N_m+j)|\asymp \frac{e^{c\sqrt{N_m}}}{\sqrt{N_m}}
  \qquad (j=-1,1,3).
\]

Finally, \eqref{eq:endpoint-size} and Corollary~\ref{cor:parity} imply
\[
  |V_1(N_m)|\ll \frac{e^{c\sqrt{N_m}}}{N_m},
  \qquad
  |V_1(N_m+2)|\ll \frac{e^{c\sqrt{N_m}}}{N_m}.
\]
For the odd neighbors, the estimates obtained above give
\[
  |V_1(N_m+j)|\gg \frac{e^{c\sqrt{N_m}}}{\sqrt{N_m}}
  \qquad (j\in\{-1,1,3\}).
\]
Since $N_m^{-1}=o(N_m^{-1/2})$, it follows that for all sufficiently large $m\in\mathcal M_-$,
\[
  |V_1(N_m)|<\min\{|V_1(N_m-1)|,|V_1(N_m+1)|\}
\]
and
\[
  |V_1(N_m+2)|<\min\{|V_1(N_m+1)|,|V_1(N_m+3)|\}.
\]
Thus both $N_m$ and $N_m+2$ are strict local minima of the sequence $|V_1(n)|$.
\end{proof}

\section{Proof of Theorem~\ref{thm:plus-family}}\label{sec:proof-plus-family}

\begin{proof}
Set
\[
  u_m:=\alpha\left(m+\frac34\right)^2=\frac{s_m^2}{c^2},
  \qquad
  y_{m,j}:=c\sqrt{M_m+j}\qquad (j\in\{-1,0,1,2,3\}).
\]
Because $m\in\mathcal M_+$, Theorem~\ref{thm:good-subseq} gives $\dist(u_m,\Z)\ge \delta_+$, so $u_m\notin\Z$. Applying Lemma~\ref{lem:surround-even} to $u_m-1$ shows that $M_m$ is the unique odd integer with $M_m<u_m<M_m+2$. Since
\[
  s_m=c\sqrt{u_m},
  \qquad
  y_{m,0}=c\sqrt{M_m},
  \qquad
  y_{m,2}=c\sqrt{M_m+2},
\]
we have $y_{m,0}<s_m<y_{m,2}$, and therefore
\[
  \delta_+\le u_m-M_m<2,
  \qquad
  \delta_+\le M_m+2-u_m<2.
\]
Since $M_m<u_m<M_m+2$, we also have $M_m=u_m+O(1)$, hence $M_m\to\infty$ with $m$. Exactly as in Lemma~\ref{lem:minus-distances}, there exist constants $v_1,v_2,v_3>0$ such that for all sufficiently large $m\in\mathcal M_+$,
\begin{equation}\label{eq:plus-endpoint-distance}
  v_1M_m^{-1/2}\le s_m-y_{m,0}\le v_2M_m^{-1/2},
  \qquad
  v_1M_m^{-1/2}\le y_{m,2}-s_m\le v_2M_m^{-1/2}.
\end{equation}
and
\begin{equation}\label{eq:plus-middle-distance}
  |y_{m,j}-s_m|\le v_3M_m^{-1/2}
  \qquad (j=-1,1,3).
\end{equation}

Writing $u=s_m+h$, we have
\[
  F_+(u)=-\sqrt 2\,(-1)^m\sin h,
  \qquad
  F_-(u)=-\sqrt 2\,(-1)^m\cos h.
\]
Hence, whenever $|h|<1$,
\begin{equation}\label{eq:Fplus-near-s}
  \sgn F_+(u)=
  \begin{cases}
    (-1)^m,& u<s_m,\\
    -(-1)^m,& u>s_m,
  \end{cases}
\end{equation}
and
\begin{equation}\label{eq:Fplus-linear}
  \frac{2\sqrt 2}{\pi}|u-s_m|\le |F_+(u)|\le \sqrt 2\,|u-s_m|,
\end{equation}
while
\begin{equation}\label{eq:Fminus-near-s}
  \sgn F_-(u)=-(-1)^m,
  \qquad
  \sqrt 2\cos(1)\le |F_-(u)|\le \sqrt 2.
\end{equation}
Using \eqref{eq:plus-endpoint-distance} and \eqref{eq:plus-middle-distance}, we may assume all the points $y_{m,j}$ lie in this range. Hence, for all sufficiently large $m\in\mathcal M_+$,
\begin{equation}\label{eq:plus-endpoint-signs}
  \sgn F_+(y_{m,0})=(-1)^m,
  \qquad
  \sgn F_+(y_{m,2})=-(-1)^m,
\end{equation}
and
\begin{equation}\label{eq:plus-endpoint-sizes}
  \frac{2\sqrt 2}{\pi}v_1M_m^{-1/2}\le |F_+(y_{m,j})|\le \sqrt 2\,v_2M_m^{-1/2}
  \qquad (j=0,2),
\end{equation}
while \eqref{eq:plus-middle-distance} and \eqref{eq:Fminus-near-s} give
\begin{equation}\label{eq:plus-middle-signs}
  \sgn F_-(y_{m,j})=-(-1)^m,
  \qquad
  \sqrt 2\cos(1)\le |F_-(y_{m,j})|\le \sqrt 2
  \qquad (j=-1,1,3).
\end{equation}
Here $y_{m,0}<s_m<y_{m,2}$, so the endpoint signs in \eqref{eq:plus-endpoint-signs} come directly from \eqref{eq:Fplus-near-s}. For $j=-1,1,3$ we only use that $|y_{m,j}-s_m|<1$, and then \eqref{eq:Fminus-near-s} gives the sign statement in \eqref{eq:plus-middle-signs}.

Let $T(n)$ again denote the main term from Corollary~\ref{cor:parity}. Since $M_m$ is odd, the parity factors in $T(M_m)$, $T(M_m+1)$, and $T(M_m+2)$ are $(-1)^{(M_m-1)/2}$, $(-1)^{(M_m+1)/2}=(-1)^{(M_m-1)/2+1}$, and $(-1)^{(M_m+2-1)/2}=(-1)^{(M_m-1)/2+1}$, respectively. Combining these with \eqref{eq:plus-endpoint-signs} and \eqref{eq:plus-middle-signs}, and using $A_0,A_1>0$, we obtain
\[
  \sgn T(M_m)=\sgn T(M_m+1)=\sgn T(M_m+2)=(-1)^{(M_m-1)/2+m}.
\]
Thus the three main terms have the same sign.

Moreover,
\[
  e^{c\sqrt{M_m+j}}=e^{c\sqrt{M_m}}\bigl(1+O(M_m^{-1/2})\bigr)
  \qquad (j=-1,0,1,2,3),
\]
so
\[
  |T(M_m)|\asymp \frac{e^{c\sqrt{M_m}}}{M_m},
  \qquad
  |T(M_m+1)|\asymp \frac{e^{c\sqrt{M_m}}}{\sqrt{M_m}},
  \qquad
  |T(M_m+2)|\asymp \frac{e^{c\sqrt{M_m}}}{M_m}.
\]
At the odd endpoints, Corollary~\ref{cor:parity} and \eqref{eq:plus-endpoint-sizes} give the error estimate
\[
  O\!\left(\frac{e^{c\sqrt{M_m}}|F_+(y_{m,j})|}{M_m}\right)
  +O\!\left(M_m^{-1/2}e^{c\sqrt{M_m}/2}\right)
  =O\!\left(\frac{e^{c\sqrt{M_m}}}{M_m^{3/2}}\right)
   +O\!\left(M_m^{-1/2}e^{c\sqrt{M_m}/2}\right)
\]
for $j=0,2$, which is $o(e^{c\sqrt{M_m}}/M_m)$. At the even middle term, Corollary~\ref{cor:parity} and \eqref{eq:plus-middle-signs} give
\begin{align*}
  &O\!\left(\frac{e^{c\sqrt{M_m}}|F_-(y_{m,1})|}{M_m}\right)
   +O\!\left(M_m^{-1/2}e^{c\sqrt{M_m}/2}\right) \\
  &\qquad =O\!\left(\frac{e^{c\sqrt{M_m}}}{M_m}\right)
   +O\!\left(M_m^{-1/2}e^{c\sqrt{M_m}/2}\right),
\end{align*}
which is $o(e^{c\sqrt{M_m}}/\sqrt{M_m})$. Writing $V_1(M_m+j)=T(M_m+j)+R(M_m+j)$, we obtain $|R(M_m+j)|=o(|T(M_m+j)|)$ for $j=0,1,2$. Hence, for all sufficiently large $m\in\mathcal M_+$,
\[
  |R(M_m+j)|<\frac12|T(M_m+j)|\qquad (j=0,1,2).
\]
Thus $\sgn V_1(M_m+j)=\sgn T(M_m+j)$ for $j=0,1,2$. Therefore
\[
  V_1(M_m),\qquad V_1(M_m+1),\qquad V_1(M_m+2)
\]
have the same sign for all sufficiently large $m\in\mathcal M_+$. Exactly the same use of Corollary~\ref{cor:parity} with \eqref{eq:plus-middle-signs} also shows that for $j\in\{-1,1,3\}$,
\[
  V_1(M_m+j)=T(M_m+j)+R(M_m+j),
  \qquad
  |R(M_m+j)|=o(|T(M_m+j)|),
\]
so in particular
\[
  |V_1(M_m+j)|\asymp \frac{e^{c\sqrt{M_m}}}{\sqrt{M_m}}
  \qquad (j=-1,1,3).
\]

Finally, \eqref{eq:plus-endpoint-sizes} and Corollary~\ref{cor:parity} give
\[
  |V_1(M_m)|\ll \frac{e^{c\sqrt{M_m}}}{M_m},
  \qquad
  |V_1(M_m+2)|\ll \frac{e^{c\sqrt{M_m}}}{M_m}.
\]
For the even neighbors, the estimates obtained above give
\[
  |V_1(M_m+j)|\gg \frac{e^{c\sqrt{M_m}}}{\sqrt{M_m}}
  \qquad (j\in\{-1,1,3\}).
\]
Since $M_m^{-1}=o(M_m^{-1/2})$, it follows that for all sufficiently large $m\in\mathcal M_+$,
\[
  |V_1(M_m)|<\min\{|V_1(M_m-1)|,|V_1(M_m+1)|\}
\]
and
\[
  |V_1(M_m+2)|<\min\{|V_1(M_m+1)|,|V_1(M_m+3)|\}.
\]
Thus both $M_m$ and $M_m+2$ are strict local minima of the sequence $|V_1(n)|$.
\end{proof}

\section{Proof of Theorem~\ref{thm:main}}\label{sec:proof-thm-main}

For completeness we record the first values displayed by Andrews.

\begin{proposition}\label{prop:initial-values}
A direct expansion of \eqref{eq:v1-def} gives
\[
\begin{aligned}
V_1(292)&=-367,\qquad V_1(293)=4,\qquad V_1(294)=375,\\
V_1(295)&=9,\qquad V_1(296)=-381,
\end{aligned}
\]
\[
\begin{aligned}
V_1(409)&=-465,\qquad V_1(410)=27,\qquad V_1(411)=473,\\
V_1(412)&=4,\qquad V_1(413)=-497,
\end{aligned}
\]
\[
\begin{aligned}
V_1(544)&=6195,\qquad V_1(545)=-18,\qquad V_1(546)=-6309,\\
V_1(547)&=-20,\qquad V_1(548)=6418,
\end{aligned}
\]
\[
\begin{aligned}
V_1(701)&=8365,\qquad V_1(702)=-273,\qquad V_1(703)=-8550,\\
V_1(704)&=-224,\qquad V_1(705)=8716.
\end{aligned}
\]
In particular, the initial values $293$, $410$, $545$, and $702$ satisfy Andrews' Conjectures~5 and~6.
\end{proposition}

\begin{proof}
This is a finite coefficient computation from \eqref{eq:v1-def}. For the range shown it is enough to truncate the outer sum in \eqref{eq:v1-def} at $n=37$, since $38\cdot 39/2>705$. The displayed equalities then show directly that for each listed $N$, the three coefficients $V_1(N)$, $V_1(N+1)$, $V_1(N+2)$ have the same sign, while both endpoints are strict local minima of $|V_1(n)|$.
\end{proof}

We can now finish the proof of Theorem~\ref{thm:main}.

\begin{proof}[Proof of Theorem~\ref{thm:main}]
By Proposition~\ref{prop:initial-values}, the integers
\[
  293,\qquad 410,\qquad 545,\qquad 702
\]
already satisfy the required same-sign and local-minimum properties.

Next, by definition,
\[
  N_m=\alpha\left(m+\frac14\right)^2+O(1),
  \qquad
  M_m=\alpha\left(m+\frac34\right)^2+O(1).
\]
The proof of Lemma~\ref{lem:alpha-not-Z} showed that $\alpha>10$. Hence
\[
  \frac{N_m}{(m+8)^2}\to \alpha,
  \qquad
  \frac{M_m}{(m+8)^2}\to \alpha,
\]
so after discarding finitely many elements from $\mathcal M_-$ and $\mathcal M_+$ we may assume simultaneously that Theorems~\ref{thm:minus-family} and \ref{thm:plus-family} apply and that
\[
  N_m>10(m+8)^2\quad(m\in\mathcal M_-),
  \qquad
  M_m>10(m+8)^2\quad(m\in\mathcal M_+).
\]
Now choose recursively $m_j\in\mathcal M_-$ and $\ell_j\in\mathcal M_+$ so that
\[
  m_j\ge 2j-1,
  \qquad
  \ell_j\ge 2j,
\]
and, with
\[
  L_{2j-1}:=N_{m_j},
  \qquad
  L_{2j}:=M_{\ell_j},
\]
we have
\[
  702<L_1<L_2<L_3<\cdots.
\]
This is possible because after each finite exclusion both admissible families still contain arbitrarily large elements. For these choices,
\[
  L_{2j-1}>10(m_j+8)^2\ge 10(2j+7)^2,
  \qquad
  L_{2j}>10(\ell_j+8)^2\ge 10(2j+8)^2.
\]
Hence $L_j>10(j+8)^2$ for every $j\ge1$.

Finally define
\[
  N_5:=293,\qquad N_6:=410,\qquad N_7:=545,\qquad N_8:=702,
\]
and for $j\ge1$ put
\[
  N_{8+j}:=L_j.
\]
Then $N_n>10n^2$ for every $n\ge 9$. Moreover, for each $n\ge9$, the three coefficients
\[
  V_1(N_n),\qquad V_1(N_n+1),\qquad V_1(N_n+2)
\]
have the same sign, and one of the numbers
\[
  |V_1(N_n)|,\qquad |V_1(N_n+1)|,\qquad |V_1(N_n+2)|
\]
is a local minimum of the sequence $|V_1(j)|$, by Theorem~\ref{thm:minus-family} or Theorem~\ref{thm:plus-family} according as $N_n=L_j$ with $j$ odd or even. 
Together with the four initial values, this proves the theorem, i.e. Andrews' Conjectures~5 and~6.
\end{proof}

\end{document}